\documentclass[11pt,a4paper]{amsart}

\usepackage[T1]{fontenc}
\usepackage[utf8]{inputenc}

\usepackage[english]{babel}
\usepackage{amsmath,amsfonts,amssymb,amsthm}

\usepackage[pdftex]{graphicx}
\usepackage{lmodern}

\usepackage{mathrsfs} %caratteri ulysse

\usepackage[bottom=3cm, top=3cm, left=3cm, right=3cm]{geometry}

\usepackage{mathtools}

\DeclareMathAlphabet{\mathbbold}{U}{bbold}{m}{n}	% nuovo alfabeto per i numeri in bold.

\theoremstyle{plain}
\newtheorem{theorem}{Theorem}
\newtheorem*{theorem*}{Theorem}
\newtheorem{prop}[theorem]{Proposition}

\newtheorem{conjecture}{Conjecture}
\newtheorem{lemma}[theorem]{Lemma}
\newtheorem*{lemma*}{Lemma}
\theoremstyle{definition}
\newtheorem{definition}[theorem]{Definition}
\theoremstyle{remark}
\newtheorem{example}[theorem]{Example}
\newtheorem{rmk}[theorem]{Remark}

\usepackage{hyperref}
\usepackage{xcolor}
\hypersetup{
    colorlinks,
    linkcolor={red!50!black},
    citecolor={blue!50!black},
    urlcolor={blue!80!black}
}

\usepackage[alphabetic, abbrev]{amsrefs}

\newcommand{\G}{\mathbb{G}}
\newcommand{\R}{\mathbb{R}}
\newcommand{\N}{\mathbb{N}}
\newcommand{\distr}{\mathcal{D}}
\newcommand{\cut}{\mathrm{cut}}
\newcommand{\Cut}{\mathrm{Cut}}

\newcommand{\Abn}{\mathrm{Abn}}

\renewcommand{\leq}{\leqslant} 
\renewcommand{\geq}{\geqslant}

\DeclareMathOperator{\tr}{\mathrm{Tr}}
\DeclareMathOperator{\codim}{\mathrm{codim}}
\DeclareMathOperator{\spn}{\mathrm{span}}
\DeclareMathOperator{\rank}{\mathrm{rank}}

\DeclareMathOperator{\diag}{\mathrm{diag}}
\DeclarePairedDelimiter{\floor}{\lfloor}{\rfloor}

\author[Luca Rizzi]{Luca Rizzi$^\sharp$}
\address{$^\sharp$ Univ. Grenoble Alpes, IF, F-38000 Grenoble, France \newline 
CNRS, IF, F-38000 Grenoble, France}
\email{\href{mailto:luca.rizzi@univ-grenoble-alpes.fr}{luca.rizzi@univ-grenoble-alpes.fr}}

\author[Ulysse Serres]{Ulysse Serres$^\flat$}
\address{$^\flat$ Univ. Lyon, Universit\'e Claude Bernard Lyon 1, CNRS, LAGEP UMR 5007, 43 bd du 11 novembre 1918, F-69100 Villeurbanne, France}
\email{\href{mailto:ulysse.serres@univ-lyon1.fr}{ulysse.serres@univ-lyon1.fr}}

\title{On the cut locus of free, step two Carnot groups}

\begin{document}

\begin{abstract}
In this note, we study the cut locus of the free, step two Carnot groups $\G_k$ with $k$ generators, equipped with their left-invariant Carnot-Carath\'eodory metric. 
In particular, we disprove the conjectures on the shape of the cut loci proposed in \cites{myasni1,myasni2} and \cite{MM}, by exhibiting sets of cut points $C_k \subset \G_k$ which, for $k \geq 4$, are strictly larger than conjectured ones. While the latter were, respectively, smooth semi-algebraic sets of codimension $\Theta(k^2)$ and semi-algebraic sets of codimension $\Theta(k)$, the sets $C_k$ are semi-algebraic and have codimension $2$, yielding the best possible lower bound valid for all $k$ on the size of the cut locus of $\G_k$.

Furthermore, we study the relation of the cut locus with the so-called abnormal set. In the low dimensional cases, it is known that
\[
\Abn_0(\G_k) = \overline{\Cut_0(\G_k)} \setminus \Cut_0(\G_k), \qquad k=2,3.
\]
For each $k \geq 4$, instead, we show that the cut locus always intersects the abnormal set, and there are plenty of abnormal geodesics with finite cut time.

Finally, and as a straightforward consequence of our results, we derive an explicit lower bound for the small time heat kernel asymptotics at the points of $C_k$.

The question whether $C_k$ coincides with the cut locus for $k\geq 4$ remains open.
\end{abstract}

\maketitle

\section{Introduction}

We recall some basic facts about sub-Riemannian manifolds and their geodesics (see \cites{nostrolibro,riffordbook} for details). Let $M$ be a smooth manifold and $\distr \subset TM$ be a smooth distribution with constant rank $k = \rank \distr$, satisfying the H\"ormander condition:
\begin{equation}\label{eq:Hormander}
\mathrm{Lie}(\Gamma(\distr))_q = T_q M, \qquad \forall q \in M,
\end{equation}
where the l.h.s. denotes the smallest Lie algebra generated by smooth sections of $\distr$. If $g$ is a smooth scalar product defined on $\distr$, then $(\distr,g)$ is a sub-Riemannian structure on $M$.

A Lipschitz path $\gamma :[0,T] \to $ is \emph{horizontal} if $\dot\gamma(t) \in \distr_{\gamma(t)}$ for a.e. $t \in [0,T]$. For any horizontal path $\gamma$, we define its length as 
\begin{equation}
L(\gamma): = \int_0^T \|\dot\gamma(t)\|_g \,dt.
\end{equation}
The length is invariant by Lipschitz reparametrization, so we can always reparametrize a horizontal curve in such a way that it has constant speed. Furthermore, we define the \emph{sub-Riemannian distance} $d : M \times M \to \R$ as $d(q,q') := \inf L(\gamma)$, where the infimum is taken over all horizontal curves that join $q$ with $q'$. Thanks to \eqref{eq:Hormander}, $d$ is finite and continuous, and so $(M,d)$ is a locally compact length metric space.

A \emph{geodesic} is a non-trivial horizontal curve $\gamma:[0,T] \to M$, with constant speed, that locally minimizes the length between its endpoints. It is \emph{maximal} if it is not the restriction of a geodesic defined on a larger interval $[0,T']$. The \emph{cut time} of a maximal geodesic is 
\begin{equation}
t_{\cut}(\gamma) := \sup\{t > 0 \mid \gamma|_{[0,t]} \text{ is a minimizing geodesic}\} > 0.
\end{equation}
Assuming that $(M,d)$ is complete, the \emph{cut locus} of $q \in M$ is the set of \emph{cut points} $\gamma(t_{\cut}(\gamma))$: 
\begin{equation}\label{eq:cutloc}
\Cut_q : = \{\gamma(t_{\cut}(\gamma)) \mid \gamma \text{ is a maximal geodesic starting at $q$}\}.
\end{equation}

The cut locus, together with the so-called abnormal geodesics, play an important role in the regularity properties of the sub-Riemannian distance \cites{cannarsarifford,RiffTrel,agrasmooth,AAA-Tangent,MM-semiconcavity} and of the heat kernel of sub-Laplacians \cites{BBN12,BBGN16,BBN-biheis}. Its properties are quite different with respect to the Riemannian setting, for example, $q$ is an accumulation point for $\mathrm{Cut}_q$, and $\mathrm{Cut}_q \cup \{q\}$ might be not closed. Moreover, while in Riemannian geometry $q' \in \mathrm{Cut}_q$ if and only if $q'$ is a critical value of the exponential map or there are two distinct minimizing geodesics joining $q$ with $q'$, this characterization is no longer true in the sub-Riemannian case due to the occurrence of abnormal geodesics.

The cut time is known explicitly for a handful of left-invariant structures on 3D Lie groups \cites{BoscRossi-Invariant,Sachkov-SE2,Sachkov-SH2}, on some Stiefel manifolds \cite{AM-Stiefel}, and for the following non-disjoint classes of Carnot groups: contact \cite{ABB12}, corank $1$ \cite{R16}, corank $2$ \cite{BBG12}, the Engel group \cite{Sachkov-Engel}, the bi-Heisenberg groups \cite{BBN-biheis} (that is, corank $1$ Carnot groups of dimension $5$), and $H$-type groups \cites{AM-Htypecut}. From the knowledge of the cut time, the cut locus can be computed via \eqref{eq:cutloc}. In a few symmetric cases, this yields an elegant and compact description for the cut locus.

\subsection{Summary of the results}

In this paper, we focus on the hierarchy $\G_k$ of \emph{free, step $2$ Carnot groups of rank $k \geq 2$}. Their first appearance traces back to the seminal works of Gaveau \cite{Gaveau77} and Brockett \cite{Brocket80}, and for this reason the corresponding minimization problem is called the \emph{Gaveau-Brockett problem} by Liu and Sussmann \cite{LiuSussman}. 

The groups $\G_k$ have the largest sets of symmetries among all Carnot groups of the same rank and step, and their geodesics can be computed quite explicitly. These features lead to expect the existence of an explicit formula describing their cut locus. This is indeed the case for $k=2$ (the Heisenberg group), where the closure of the cut locus is described succinctly as the zero set of an algebraic function. Surprisingly, such an algebraic description also exists for the free step $2$ Carnot group of rank $3$ -- the so-called $(3,6)$ group -- as proven independently and with different strategies in \cite{myasni1} and \cite{MM}.

Extrapolating from the known results in these cases, in \cites{myasni1,myasni2,MM}, the authors conjectured precise algebraic formulas for the cut loci of $\G_k$, for all $k \geq 2$. In this note, we disprove these conjectures. In particular, we exhibit an explicit hierarchy of sets $C_k \subset \G_k$ of cut points (see Definition~\ref{d:sigmastar}), which coincide with the corresponding cut loci for $k =2,3$, but are strictly larger than the conjectured ones for $k \geq 4$. While the previously conjectured cut loci were, respectively, smooth algebraic sets of codimension $\Theta(k^2)$ \cites{myasni1,myasni2} and algebraic sets of codimension $\Theta(k)$ \cite{MM}, the set $C_k$ is semi-algebraic of codimension $2$, for all $k$ (see Theorem~\ref{t:main}).

Furthermore, we study the relation between the cut locus and the so-called abnormal set. In the low dimensional cases, it is known that
\begin{equation}
\Abn_0(\G_k) = \overline{\mathrm{Cut}_0(\G_k)} \setminus \Cut_0(\G_k), \qquad k=2,3.
\end{equation}
Instead, for each $k \geq 4$, we show that the cut locus always intersects the abnormal set, and that there are plenty of abnormal geodesics with finite cut time (see Proposition~\ref{p:abnormal}).

Finally, and as a straightforward consequence of our results, we obtain an explicit lower bound for the small time heat kernel asymptotics at the points of $C_k$ (see Theorem~\ref{t:heatkernel}).

\section{Free step 2 Carnot groups of rank k}

We refer to \cites{Montgomerybook,Jeanbook}, for the definition of Carnot group. Here, we only deal with the specific free, step $2$ case. Let $\G_k := \R^k \oplus \wedge^2\R^k$. We identify $\wedge^2\R^k$ with the vector space of skew-symmetric real matrices, that is $v \wedge w = v w^* - w v^*$ for $v,w \in \R^k$. We denote points $(x,Y) \in \G_k$, where $x \in \R^k$ and $Y$ is a skew-symmetric matrix. The free, step $2$ Carnot group of rank $k$ is the sub-Riemannian structure on $\G_k$ generated by the set of global orthonormal vector fields:
\begin{equation}\label{eq:generators}
X_i := \partial_{x_i} -\frac{1}{2}\sum_{1\leq \ell<m\leq k}(e_i \wedge x)_{\ell m} \partial_{Y_{\ell m}}, \qquad i=1,\dots,k,
\end{equation}
where $\{e_1,\dots,e_n\}$ is the standard basis of $\R^k$. More precisely, the horizontal distribution is defined by $\distr := \spn\{X_1,\dots,X_k\}$ and the sub-Riemannian metric by $g(X_i,X_j) = \delta_{ij}$.

For all $i <j$, we have  $[X_i,X_j] = \partial_{Y_{ij}}$. In particular, the vector fields \eqref{eq:generators} generate the free, nilpotent Lie algebra of step $2$ with $k$ generators:
\begin{equation}
\mathfrak{g} = \mathfrak{g}_1 \oplus \mathfrak{g}_2, \qquad \text{where} \qquad \mathfrak{g}_1 = \spn\{X_1,\dots,X_k\}, \quad \mathfrak{g}_2 = \spn \{ \partial_{Y_{ij}}\}_{i<j}.
\end{equation}
There exists a unique Lie group structure on $\G_k$ such that the vector fields $X_i$ are left-invariant, given by the polynomial product law
\begin{equation}
(x,Y) \star (x',Y') = \left(x+x',Y + Y' + \frac{1}{2} x \wedge x'\right).
\end{equation}

Thus, $(\G_k,\star)$ is a connected, simply connected Lie group of dimension $k(k+1)/2$, such that its Lie algebra $\mathfrak{g}$ of left-invariant vector fields is isomorphic to the free, nilpotent, stratified Lie algebra of step $2$ with $k$ generators, and such that the first stratum $\mathfrak{g}_1$ is equipped with a left-invariant scalar product. Any Lie group $(\G'_k,\star')$ with the same properties is isomorphic to $(\G_k,\star)$ and they carry isometric sub-Riemannian structures.

Carnot groups are equipped with a one-parameter family of dilations. For $\G_k$, it is given by $\delta_\varepsilon(x,Y) := (\varepsilon x, \varepsilon^2 Y)$, for $\varepsilon>0$. As a consequence of this fact, the metric spaces $(\G_k,d)$ are complete, and there exists a minimizing geodesic joining any given pair of points.

\begin{example}
The case $k=2$ is the well-known Heisenberg group. Indeed, we can identify $(x,Y) \in \R^2 \oplus \wedge^2 \R^2$ with $(x,z) \in \R^2 \oplus \R$, so that the generating vector fields \eqref{eq:generators} read
\begin{equation}
X_1 = \partial_{x_1} - \frac{x_2}{2}  \partial_z, \qquad X_2 = \partial_{x_2} + \frac{x_1}{2}  \partial_z.
\end{equation}
\end{example}

\begin{example}
The case $k=3$ can be dealt with by identifying $(x,Y) \in \R^3 \oplus \wedge^2\R^3$ with $(x,t) \in \R^3 \oplus \R^3$. More precisely, any $3\times 3$ skew-symmetric matrix can be written as $Y = v \wedge w$ (in a non-unique way), and is identified with the cross product $t=v \times w$. Under this identification, the tautological action of $Y$ on $\R^3$ reads
\begin{equation}
Y x = (v \wedge w) x = x \times (v \times w) = x \times t, \qquad \forall x \in \R^3,
\end{equation}
and the generating vector fields \eqref{eq:generators} are
\begin{equation}
X_1 = \partial_{x_1} + \frac{x_3}{2} \partial_{t_2} - \frac{x_2}{2} \partial_{t_3}, \quad X_2 = \partial_{x_2} + \frac{x_1}{2} \partial_{t_3} -\frac{x_3}{2} \partial_{t_1}, \quad X_3 = \partial_{x_3} + \frac{x_2}{2} \partial_{t_1}- \frac{x_1}{2} \partial_{t_2}.
\end{equation}
\end{example}

\subsection{Geodesics}  In this case, a Lipschitz path $\gamma :[0,T] \to \G_k$ is horizontal if there is a control $u \in L^\infty([0,T],\R^k)$ such that, for almost every $t \in [0,T]$, we have
\begin{equation}\label{eq:horizontal}
\dot\gamma(t) = \sum_{i=1}^k u_i(t) X_i(\gamma(t)).
\end{equation}

The standard method to solve the length minimization problem is based on the Pontryagin maximum principle. In the step $2$ case, thanks to Goh condition, one can rule out all the so-called strictly abnormal curves. This analysis (see e.g.\ \cite[Appendix A]{LiuSussman}) yields that geodesics are all the horizontal curves $\gamma(t) = (x(t),Y(t))$ with control
\begin{equation}\label{eq:control}
u(t) = e^{-t \Omega} p, \qquad (p,\Omega) \in \R^k \oplus \wedge^2\R^k.
\end{equation}
In particular, all geodesics are real-analytic curves and can be extended on the maximal interval $[0,+\infty)$. The pair $(p,\Omega)$ is also referred to as the \emph{initial covector} in the Hamiltonian formalism.  Taking into account the explicit fields \eqref{eq:generators}, for any fixed $(p,\Omega)$, the actual geodesic can be recovered integrating
\begin{equation}\label{eq:geods}
\dot{x}(t)  =  u(t), \qquad \dot{Y}(t)  = \frac{1}{2} x(t) \wedge u(t). 
\end{equation}
\begin{rmk}\label{r:gauthier} We point out a useful observation from \cite[Prop. 3]{Gauthier-trick}, which follows by putting $\Omega$ in real normal form. For any control $u(t)$ as in $\eqref{eq:control}$, there is a unique $\tilde{\Omega} \in \wedge^2 \R^k$ with all simple non-zero eigenvalues and largest possible kernel such that $e^{-t\Omega}p = e^{-t \tilde\Omega}p$.
\end{rmk}
Most of the progress in the study of free Carnot groups of low dimension ($k=2,3$) is due to the fact that \eqref{eq:geods} can be explicitly integrated in terms of trigonometric functions. Geodesics for the case $k=2$, corresponding to the Heisenberg group, are well known. For the case $k=3$, explicit formulas appeared first in \cite{myasni1}. An explicit integration of the geodesic flow for the general case $\G_k$ appeared first in \cite[Thm. 4.1]{myasni2}. However, if the latter were correct, it would imply that the geodesic flow can only reach points $(x,Y)$ with $\rank( Y) = 2$. To our best knowledge, the only complete and correct integration of the geodesic flow for general $\G_k$ appears in \cite{MPAM-geodesics}, in terms of the spectral projectors of $\Omega$, together with explicit worked out examples for all $k\leq 5$.

\section{The conjectured cut loci}\label{s:cut}

By the left-invariance of the sub-Riemannian structure on $\G_k$, one can recover $\mathrm{Cut}_q(\G_k)$ for any $q \in \G_k$ by left translation of $\mathrm{Cut}_0(\G_k)$, where $0=(0,0)$ is the identity of $\G_k$. In the following, with the term ``cut locus'', we will refer to $\Cut_0(\G_k)$.

The cut locus of $\G_2$ (the Heisenberg group) is well known, and consists in the set of points $(0,Y) \in \R^2 \oplus \wedge^2\R^2$, with $Y \neq 0$. The computation of the cut time and cut locus for the case $k=3$ is much harder, and requires a careful manipulation of geodesic equations and symmetries. This has been done independently in \cite{myasni1} and \cite{MM}. There, it was proved that, for the step $2$ free Carnot groups of rank $k=2,3$, one has
\begin{equation}\label{eq:provedcut}
\Cut_0(\G_k) = \{ (x,Y) \mid Y = v\wedge w \neq 0,\, Yx = 0 \}, \qquad k=2,3.
\end{equation}
The characterization in \eqref{eq:provedcut} being dimension-free, the authors were naturally led to two closely related conjectures.
\begin{conjecture}[\cites{myasni1,myasni2}]\label{c:conj1}
The cut locus of step $2$, free Carnot group of rank $k$ is
\begin{equation}
P_k := \{ (x,Y) \mid Y = v\wedge w \neq 0,\, Yx = 0 \}.
\end{equation}
\end{conjecture}
\begin{conjecture}[\cite{MM}]\label{c:conj2}
The cut locus of step $2$, free Carnot group of rank $k$ is
\begin{equation}
\Sigma_k := \{ (x,Y) \mid Y \neq 0,\, Yx = 0 \}.
\end{equation}
\end{conjecture}
Clearly, $\Sigma_k = P_k$ for $k=2,3$, but $P_k \subsetneq \Sigma_k$ for $k \geq 4$. These conjectures imply analogous claims on the size of the cut loci, in particular $\codim(P_k) = \Theta(k^2)$ and $\codim(\Sigma_k) = \Theta(k)$, for large $k$ (see the forthcoming Proposition~\ref{p:dimsigma}).

\subsection{Semi-algebraic sets}
A set $A \subset \R^n$ is \emph{semi-algebraic} if it is the result of a finite number of unions and intersections of sets of the form $\{f = 0\}$, $\{g > 0\}$, where $f$, $g$ are polynomials on $\R^n$. We recall some of their basic properties, referring to \cite{realalgebraicgeometry} for details. If $A$ is smooth, its dimension as a semi-algebraic set \cite[Def. 2.8.1]{realalgebraicgeometry} is equal to its dimension as a smooth manifold. Moreover, if $A$ is the finite union of semi-algebraic sets $A_1,\dots,A_p$, then $\dim(A) = \max_i \dim(A_i)$. Finally, if $A$ is a semi-algebraic set, then its closure, its interior and any Cartesian projection of $A$ are semi-algebraic sets.

\begin{prop}\label{p:dimsigma}
$P_k \subseteq \G_k$ is a smooth semi-algebraic set of codimension $(k^2-5k+10)/2$. $\Sigma_k \subseteq \G_k$ is a semi-algebraic set of codimension $2\floor{k/2}$.
\end{prop}
\begin{proof}
It is clear that both sets are semi-algebraic. Notice also that
\begin{equation}
\Sigma_k = \bigsqcup_{r=1}^{\floor{k/2}} \Sigma_k^{2r}, \qquad \Sigma_k^{2r} = \{(x,Y)\mid Yx =0, \, \rank(Y) = 2r\}.
\end{equation}
Indeed $P_k = \Sigma_k^2$. We now prove that each $\Sigma_k^{2r}$ is a smooth manifold and we compute its codimension. The set $S_{2r} \subset \wedge^2 \R^k$ of rank $2r$ skew-symmetric matrices is a smooth submanifold of codimension $\dim(\mathfrak{so}(k-2r))=(k-2r)(k-2r-1)/2$. Let $\phi_{2r}: \R^k \times S_{2r} \to \R^k$ be the smooth map $\phi_{2r}(x,Y) = Yx$. Then, $\Sigma_k^{2r} = \phi^{-1}_{2r}(0)$. 

We claim that $\phi_{2r}$ has rank $2r$. This claim concludes the proof since $\phi^{-1}_{2r}(0)$ is then a codimension $2r$ submanifold of a codimension $(k-2r)(k-2r-1)/2$ submanifold of $\G_k$, that is a submanifold of codimension $(k-2r)(k-2r -1)/2+2r = (k^2- (4 r+1)k +2 r (2 r+3))/2$. To prove the claim, we drop the subscript from the notation $\phi_{2r}$, and we assume, without loss of generality, that $(x,Y) \in \Sigma_k^{2r}$ is of the form
\begin{equation}\label{eq:15}
Y = \begin{pmatrix}
\bar{Y} & \\
& \mathbbold{0}_{k-2r}
\end{pmatrix}, \qquad x =\begin{pmatrix} 0 \\ x_0 \end{pmatrix}, \qquad \rank (\bar{Y}) = 2r,\quad x_0 \in \R^{k-2r}.
\end{equation}
The smooth variations $\gamma_i(\varepsilon)=(x+\varepsilon e_i,Y) \in \R^k \times S_{2r}$, for $i=1,\dots,2r$, are such that
\begin{equation}
\partial_t \phi(\gamma_i)(0) = \bar{Y} e_i.
\end{equation}
Thus, the image of the differential of $\phi$ at $(x,Y)$ contains $\spn\{e_1,\dots,e_{2r}\}$, and $\rank(\phi) \geq 2r$. Suppose that $\rank(\phi) > 2r$. Therefore, there exists $\gamma(t) = (x+t x',Y+t Y') \in \R^k \times S_{2r}$, such that $\partial_t \phi(\gamma(0))$ and $\spn\{e_1,\dots,e_{2r}\}$ are independent. Splitting $Y'$ as in \eqref{eq:15},
\begin{equation}
Y' = \begin{pmatrix}
Y'_1 & Y'_2 \\
Y'_3 & Y'_4
\end{pmatrix},\qquad  \text{we obtain} \qquad \partial_t \phi(\gamma(0)) = \begin{pmatrix} * \\ Y'_4 x_0 \end{pmatrix},
\end{equation}
where the latter is independent of $\spn\{e_1,\dots,e_{2r}\}$. But this means that, for sufficiently small $t$, we have $\rank (Y+t Y') >2r$, which contradicts the fact that $\gamma \in \R^k \times S_{2r}$.
\end{proof}

\section{A larger set of cut points}\label{s:larger}
The orthogonal group $\mathrm{O}(k)$ acts smoothly on $\G_k$. In particular, any $M \in \mathrm{O}(k)$ induces an isometry of Carnot groups $\rho_M : \G_k \to \G_k$, given by $\rho_M(x,Y) = (Mx, MY M^*)$.

\begin{definition}\label{d:sigmastar}
We define $C_k$ as the set of points $(x,Y) \in \R^k \oplus \wedge^2\R^k$, with $Y \neq 0$, such that there exists a non-trivial $M \in \mathrm{O}(k)$ stabilizing $(x,Y)$, i.e.\
\begin{equation}
Mx = x, \qquad MYM^* = Y,
\end{equation}
and such that $M|_{\ker Y} = \mathbbold{1}$.
\end{definition}
For the Heisenberg group, $C_2 = \Sigma_2 = P_2 = \{(0,Y)\mid Y  \neq 0 \}$. For rank $k=3$, we can write $Y= v \wedge w$ for some (non-unique) pair $v,w \in \R^k$. In this case, any $M \in \mathrm{O}(3)$ such that $MYM^* = Y$ is an orthogonal transformation in the plane $\spn\{v,w\} = (\ker Y)^\perp$. By the last condition, such a rotation is non-trivial. Thus, the condition $Mx = x$ implies that $x \perp \spn\{v,w\}$, and $C_3 = \Sigma_3 = P_3$. As it will be evident from the forthcoming discussion about normal forms, this is no longer true for $k\geq 4$. We now clarify the shape of $C_k$.

\subsection{Normal forms}\label{s:normalform}
Let $(x,Y) \in \G_k$. Let $0 <\alpha_1 < \dots < \alpha_\ell$  be the absolute values of the non-zero distinct eigenvalues of $Y$, each having multiplicity $m_i$, and $m_0 = \dim (\ker Y)$. Up to an isometry $\rho_O$, we can assume that $Y$ has the real normal form
\begin{equation}
Y = \begin{pmatrix}
\alpha_1 \mathbbold{J}_{2m_1} & & &  \\
& \ddots & &  \\
& & \alpha_\ell \mathbbold{J}_{2m_\ell} & \\
& & & \mathbbold{0}_{m_0}
\end{pmatrix}, \qquad \mathbbold{J}_{2m} = \begin{pmatrix}
0 & \mathbbold{1}_{m} \\
-\mathbbold{1}_{m} & 0
\end{pmatrix}.
\end{equation}
Accordingly, we write $x = (x_1, \dots, x_\ell, x_0)$, with $x_i \in \R^{2m_i}$ for $i=1,\dots,\ell$, and $x_0 \in \R^{m_0}$. 

If we insist that $(x,Y) \in C_k$, the conditions $M|_{\ker Y} = \mathbbold{1}$ and $M YM^* = Y$ imply that
\begin{equation}\label{eq:eccola}
M = \begin{pmatrix}
M_1 & & & \\
& \ddots & & \\
& & M_\ell & \\
& & & \mathbbold{1}_{m_0}
\end{pmatrix},  \qquad M_i \in \mathrm{O}(2m_i) \cap \mathrm{Sp}(2m_i) \simeq \mathrm{U}(m_i).
\end{equation}
Thus, the conditions $Mx = x$, and $M \neq \mathbbold{1}$ yield that $(x,Y) \in C_k$ if and only if, for some $i \in \{1,\dots,\ell\}$, either $x_{i} =0$ or $m_{i} > 1$. In the first case, there is no further restriction on $M_i$, while in the second we must have $M_i x_i = x_i$. Thus, if $(x,Y) \in C_k$, the set
\begin{equation}
\mathscr{M}(x,Y) := \{ M \in \mathrm{O}(k)  \text{ satisfying the property of Definition~\ref{d:sigmastar}}\} \cup \{\mathbbold{1}_k\},
\end{equation}
is a non-trivial Lie subgroup of $\mathrm{O}(k)$ of dimension
\begin{equation}\label{eq:dimM}
\dim\mathscr{M}(x,Y) = \sum_{i \in I_0} \dim \mathrm{U}(m_i) + \sum_{i \notin I_0} \dim \mathrm{U}(m_i-1),
\end{equation}
where $I_0 \subseteq \{1,\dots,\ell \}$ is the set of indices $i$ such that $x_i = 0$. With this notation, we have the following.

\begin{prop}\label{p:family}
For any geodesic $\gamma$ joining the origin with $(x,Y) \in C_k$, there exists a non-empty family of mutually distinct geodesics with the same endpoints and length, obtained by the action of $\rho_M$ on $\gamma$, with $M \in \mathscr{M}(x,Y)$. The dimension of such a family is
\begin{equation}\label{eq:numberofparam}
N(x,Y) = \sum_{i\in I_0} m_i^2+\sum_{i \notin I_0} (m_i-1)^2 >0.
\end{equation}
\end{prop}
\begin{proof}
For all $M \in \mathscr{M}(x,Y)$, the curves $\rho_{M}(\gamma)$ are geodesics joining the same endpoints and with the same length of $\gamma$. By contradiction, assume that $\rho_{M_1}(\gamma(t)) = \rho_{M_2}(\gamma(t))$ for all $t \in [0,T]$ and for $M_1 \neq M_2 \in \mathscr{M}(x,Y)$. In particular, $\mathbbold{1}\neq M_2^* M_1 \in \mathscr{M}(x,Y)$. Let $E \subsetneq \R^n$ be the eigenspace of $M_2^* M_1$ corresponding to the eigenvalue $1$. Since $M_1 x(t) = M_2 x(t)$, and $\dot{x}(t) = u(t)$, both $u(t),x(t) \in E$ for all $t \in [0,1]$. The geodesic equations \eqref{eq:geods} imply
\begin{equation}\label{eq:key}
Y(t) \in \wedge^2 E \subsetneq \wedge^2\R^k.
\end{equation}
Let $v \perp E$, with $v \neq 0$. By \eqref{eq:key}, $v \in \ker Y$. By Definition~\ref{d:sigmastar}, however, $M_2^*M_1 v = v$, hence $v \in E$ and $v =0$. This is a contradiction, so the $\rho_{M}(\gamma)$ are mutually distinct. The dimension of such a family is $N(x,Y) = \dim \mathscr{M}(x,Y)$ given by \eqref{eq:dimM}.
\end{proof}
\begin{rmk}
By construction, the family in Proposition~\ref{p:family} is certainly composed of distinct geodesics obtained from a given one by isometric transformations. This might not exhaust all geodesics joining the origin with a given point $(x,Y) \in C_k$. For example, in corank $1$ Carnot groups there are continuous families of non-isometric geodesics joining the origin with a given (non-generic) point, all having the same length \cite{LR-howmany}.
\end{rmk}

\begin{prop}\label{p:dimC}
For all $k \geq 2$, the set $C_k \subset \G_k$ is semi-algebraic of codimension $2$.
\end{prop}
\begin{proof} 
By Definition~\ref{d:sigmastar}, $(x,Y) \in C_k$ if and only if there exist $M \in \mathrm{O}(k)$ stabilizing $(x,Y)$, a matrix $W \in \R^{k \times k}$ whose columns generate $\ker Y$, and $M|_{\ker Y} = \mathbbold{1}$, that is $MW =W$. Thus, $C_k$ is the projection on the first two factors of the set $V_k \subset \R^k \times \wedge^2 \R^k \times \mathrm{O}(k) \times \R^{k \times k}$, defined by the following formula:
\begin{multline}
V_k := \left\lbrace(x,Y,M,W) \mid Mx = x,\quad MYM^* = Y, \quad \right. \\
\left. \rank (W) = k-\rank(Y), \quad Y W = 0, \quad MW = W \right\rbrace.
\end{multline}
The rank conditions being semi-algebraic, $V_k$ is a semi-algebraic set, and so is $C_k$.

To compute the dimension of $C_k$, we follow a strategy inspired by \cite[App. A]{AGL-horizontal}. We decompose $C_k$ as the disjoint union of a finite number of smooth semi-algebraic sets, each one containing points $(x,Y)$ with the same ``type'' of normal form as above:
\begin{equation}\label{eq:decomposition}
C_k = \bigsqcup_{r=1}^{\floor{k/2}}\left( \bigsqcup_{\ell=1}^{r-1}\bigsqcup_{m_1+\dots+m_\ell = r} C_k^{2r|m_1,\dots,m_\ell} \bigsqcup_{I_0  \neq \emptyset } C_k^{2r|1,\dots,1|I_0}\right).
\end{equation}
Each component in \eqref{eq:decomposition} is labeled by the rank $2r$ of $Y$, and the multiplicities $m_1,\dots,m_\ell$ of its non-zero distinct eigenvalues, ordered according to their absolute values $0<\alpha_1<\dots<\alpha_\ell$. In the special case of all simple eigenvalues $\ell=r$, that is $m_1=\dots=m_r=1$, we further decompose according to the set $I_0 \subseteq \{1,\dots,r\}$ such that $x_i =0$ for $i \in I_0$, and $x_i \neq 0$ for $i \notin I_0$. This decomposition exhausts all points of $C_k$.

We now prove that the component $C_k^{2r|1,\dots,1|I_0}$, corresponding to the case in which $\rank (Y)=2r = 2\floor{k/2}$ is maximal, with simple eigenvalues, and where $I_0 = \{ i\}$ is a singleton, is smooth of codimension $2$. A similar argument, which we omit, proves that all the other components in the decomposition are smooth, and with larger codimension.

Assume, for simplicity, that $k$ is even and that $I_0 = \{1\}$. Let $\R^n_* = \R^n \setminus \{0\}$ and $\R^{n}_{\text{ord}}$ denote the open subset of ordered, distinct $n$-tuples $0< \alpha_1<\dots< \alpha_{n}$. Consider the map $\Phi: \R^{k/2}_{\text{ord}} \times \R^{k/2-1}_* \times \mathrm{O}(k)/\mathrm{U}(1) \to C_k^{2k|1,\dots,1|I_0}$ given by
\begin{equation}
\Phi(\alpha_1,\dots,\alpha_{k/2}, b_2,\dots,b_{k/2},M)  = \rho_M\left(\begin{pmatrix}
0 \\
b_2 \\
\vdots \\
b_{k/2}
\end{pmatrix}, \begin{pmatrix}
\alpha_1 \mathbbold{J}_{2} & & & \\
& \alpha_2 \mathbbold{J}_{2} & & \\
& & \ddots & \\
& & & \alpha_{k/2} \mathbbold{J}_{2} \\
\end{pmatrix}\right),
\end{equation}
where $\mathrm{O}(k)/\mathrm{U}(1)$ is the homogeneous space (a smooth manifold) of left cosets $MU$, where $U \in \mathrm{U}(1) \subset \mathrm{O}(k)$ is the subgroup that stabilizes the given normal form of $(x,Y)$, that is
\begin{equation}
U = \begin{pmatrix}
M_1 & & & \\
& \mathbbold{1}_2 & & \\
& & \ddots & \\
& & & \mathbbold{1}_2
\end{pmatrix}, \qquad M_1 \in \mathrm{O}(2) \cup \mathrm{Sp}(2) \simeq \mathrm{U}(1).
\end{equation}
The map $\Phi$ is a bijection. This gives $C_k^{2k|1,\dots,1|I_0}$ a smooth structure, with dimension
\begin{equation}
\frac{k}{2} +\left(\frac{k}{2}-1\right) + \dim \mathrm{O}(k)/\mathrm{U}(1) =  \frac{k(k+1)}{2}-2 .
\end{equation}
The same argument yields the same formula for odd $k$. Thus, since $\dim(\G_k) = k(k+1)/2$, the codimension of the larger component of $C_k$ is $2$. Notice that, for $k=2,3$, there is only one component in \eqref{eq:decomposition}, hence in these cases $C_k$ is a \emph{smooth} semi-algebraic set.
\end{proof}

The next statement holds true, in particular, for all step $2$ Carnot groups.

\begin{lemma}\label{l:maxincut}
Let $(\distr,g)$ be a complete sub-Riemannian structure on $M$. Assume that all minimizing curves are real-analytic. If there exist two distinct minimizing geodesics joining $q$ with $q'$, then $q' \in \Cut_q$.
\end{lemma}
\begin{proof}
Let $\gamma_i :[0,T] \to M$, $i=1,2$, be distinct minimizing geodesics joining $q$ with $q'$. In local coordinates, $\exists n \in \mathbb{N}$ such that their $n$-th derivatives at $T$ are different. Assume that we can extend $\gamma_1$ to a minimizing geodesic on the interval $[0,T+\varepsilon]$, with $\varepsilon > 0$. Then, the curve $\gamma: [0,T+\varepsilon] \to M$ such that $\gamma|_{[0,T]} = \gamma_2$, and $\gamma|_{[T,T+\varepsilon]} = \gamma_1$ is minimizing on the interval $[0,T+\varepsilon]$, but not real-analytic. This is a contradiction, then $q' \in \mathrm{Cut}_q$.
\end{proof}
\begin{rmk}
The conclusion of Lemma~\ref{l:maxincut} holds true, although with a different proof, if the real-analytic hypothesis is replaced by the assumption that there are no abnormal minimizing curves. Otherwise, one cannot a priori exclude the case of ``branching'' minimizing curves.
More precisely, one can prove that the segment $\gamma|_{[T,T+\varepsilon]}$ defined as in the proof of Lemma~\ref{l:maxincut} must be an abnormal minimizing curve.
\end{rmk}

It is now clear that the set $C_k$ of Definition~\ref{d:sigmastar} is a set of cut points, which for $k\geq 4$ is strictly larger than the sets $P_k$ and $\Sigma_k$ of Conjectures~\ref{c:conj1} and~\ref{c:conj2}.

\begin{theorem}\label{t:main}
The following chain of inclusions holds true, for all $k \geq 4$:
\begin{equation}
P_k \subsetneq \Sigma_k \subsetneq C_k \subseteq \Cut_0(\G_k).
\end{equation}
In particular, if $\mathrm{Cut}_0(\G_k)$ is semi-algebraic, $\mathrm{codim}(\Cut_0(\G_k)) \leq 2$, for all $k$.
\end{theorem}
\begin{proof}
The first strict inclusion is trivial. The second is a consequence of the normal forms in Section~\ref{s:normalform}. Notice that all minimizing curves of $\G_k$ are real-analytic. Hence, the third inclusion follows from Proposition~\ref{p:family}, applied to minimizing geodesics, and Lemma~\ref{l:maxincut}. The claim on the codimension follows from Proposition~\ref{p:dimC}.
\end{proof}

\subsection{Relation with the abnormal set}
Let $\mathrm{End}: L^\infty([0,1],\R^k) \to \G_k$ be the \emph{end-point map}, which associates with a given control $u(t)$ the end-point of the corresponding horizontal curve \eqref{eq:horizontal}. Trajectories corresponding to critical points of $\mathrm{End}$ are called \emph{singular curves}, and the set $\Abn_0(\G_k)$ of critical values of $\mathrm{End}$ is the so-called \emph{abnormal set}. Determining the ``size'' of $\Abn_0(\G_k)$ on a general sub-Riemannian manifold is one of the hard open problems in the field \cite{Agraproblems}. A computation using the formula for the differential of the end-point map and the chronological calculus (see \cite{nostrolibro}) yields
\begin{equation}\label{eq:abnormalset}
\Abn_0(\G_k) = \bigsqcup_{W \in \mathrm{Gr}(k,k-2)} W \oplus \wedge^2 W, \qquad \forall k \geq 2,
\end{equation}
where $\mathrm{Gr}(k,k-2)$ is the Grassmannian of codimension $2$ subspaces of $\R^k$. We refer to \cite{EnricoSard} for a proof of \eqref{eq:abnormalset} not resorting to chronological calculus, and several alternative characterizations. We point out that, if $\dim(\ker Y) \geq 3$, then $(x,Y) \in \Abn_0(\G_k)$.

In the low dimensional cases, it turns out that the boundary of the cut locus coincides with the abnormal set:
\begin{equation}\label{eq:closure}
\Abn_0(\G_k) = \overline{\Cut_0(\G_k)} \setminus \Cut_0(\G_k), \qquad k=2,3.
\end{equation}
For $k\geq 4$, formula \eqref{eq:closure} is no longer true. More precisely, we have the following result, yielding also the first example of abnormal geodesic with finite cut time.
\begin{prop} \label{p:abnormal}
For all $k \geq 2$, we have 
\begin{equation}
\overline{C}_k \setminus C_k \subseteq \Abn_0(\G_k) \subsetneq \overline{C}_k.
\end{equation}
For each $k \geq 4$, the first inclusion is strict, and thus $\mathrm{Cut}_0(\G_k) \cap \Abn_0(\G_k) \neq \emptyset$. Moreover, there exist abnormal geodesics with finite cut time.
\end{prop}
\begin{proof}
Let $(x,Y) \in \overline{C}_k \setminus C_k$. By the discussion in Section~\ref{s:normalform}, we assume that
\begin{equation}\label{eq:sequence}
x = \begin{pmatrix}
x_1 \\
\vdots \\
x_\ell \\
x_0
\end{pmatrix}, \qquad Y = \begin{pmatrix}
\alpha_1 \mathbbold{J}_2 & & & \\
 & \ddots & & \\
 & & \alpha_\ell \mathbbold{J}_2 & \\
 & & & \mathbbold{0}_{m_0}
\end{pmatrix}, \qquad 0 \neq x_i \in \R^2, \quad i =1,\dots,\ell.
\end{equation}
Let $(x^n,Y^n) \in C_k$ be a sequence converging to $(x,Y)$ as $n \to +\infty$. Since the non-zero eigenvalues of $Y$ are simple, the same holds  for the non-zero eigenvalues of $Y^n$. Consequently, the associated real projectors $\pi_i^n$, for $i=1,\dots,\ell$, are continuous \cite[Chap. 2, Thm. 5.1]{Kato}, and so are the projections $\pi_\perp^n := \mathbbold{1} -\sum_{i=1}^\ell \pi_i^n$. Observe that $\lim_n \pi_\perp^n = \pi_0$ is the projection on $\ker Y$, but $\pi_\perp^n$ is not necessarily the projection on $\ker Y^n$ (a pair of eigenvalues $\pm i \alpha^n \neq 0$ of $Y^n$ could coalesce to zero for $n \to +\infty$). In particular, we stress that $m_0=\dim(\ker Y) \geq 2$, otherwise no such a sequence could exist. 

If $x_0 = 0$, then $(x,Y) \in \Abn_0(\G_k)$ by the characterization \eqref{eq:abnormalset} of the abnormal set. Thus, let $x_0 \neq 0$. By our assumption on the sequence $(x^n,Y^n)$, we have $\pi_i^n(x^n)  \neq 0$, for all $i \in \{1,\dots,\ell,\perp\}$, and large $n$. Since $(x^n,Y^n) \in C_k$, and by the characterization of Section~\ref{s:normalform}, we have $m_0 = \dim(\ker Y) \geq 3$. In this case, $(x,Y) \in \Abn_0(\G_k)$, by \eqref{eq:abnormalset}. This proves that $\overline{C}_k \setminus C_k \subseteq \Abn_0(\G_k)$.

We now prove that $\Abn_0(\G_k) \subsetneq \overline{C}_k$. By the characterization of \eqref{eq:abnormalset}, let $(x,Y)  \in W \oplus \wedge^2 W$, where $W \subset \R^k$ is a codimension $2$ subspace. Up to an orthogonal transformation, we can assume that $W = \{(x_1,\dots,x_{k-2},0,0)^* \in \R^k\} \simeq \R^{k-2}$, that is
\begin{equation}
x = \begin{pmatrix}
x_1 \\
0
\end{pmatrix}, \qquad Y = \begin{pmatrix}
Y_1 & \\
& \mathbbold{0}_2
\end{pmatrix}, \qquad x_1 \in W, \quad Y_1 \in \wedge^2W.
\end{equation}
Let $(x^n,Y^n)$ be the sequence with $x^n = x$, and $Y^n$ obtained from $Y$ by replacing the lower right block $\mathbbold{0}_2$ with $\tfrac{1}{n} \mathbbold{J}_2$.
Then $(x^n,Y^n) \in C_k$ for all $n$, and the sequence converges to the given abnormal point $(x,Y)$. This proves $\Abn_0(\G_k) \subseteq \overline{C}_k$. The inclusion is strict, for all $k \geq 2$, since $(x,Y) \in C_k$ does not belong to $\Abn_0(\G_k)$ if $Y$ has maximal rank.

To conclude the proof, for each $k \geq 4$, we give an example of cut point $(x,Y) \in C_k \cap \Abn_0(\G_k)$, reached by an abnormal minimizing geodesic. Let
\begin{equation}
x = 0, \qquad  Y= \begin{pmatrix}
\alpha \mathbbold{J}_2 &  \\
& \mathbbold{0}_{k-2}
\end{pmatrix}, \qquad \alpha > 0.
\end{equation}
It is not hard to prove that the geodesics $\gamma:[0,1] \to \G_k$ with control $u(t) = e^{-t\Omega}p$, with
\begin{equation}
p = \begin{pmatrix}
p_0 \\ 0
\end{pmatrix}, \qquad \Omega= \begin{pmatrix}
2 \pi\mathbbold{J}_2 &  \\
& \mathbbold{0}_{k-2}
\end{pmatrix}, \qquad p_0 \in \R^2,\quad |p_0|^2 = 4\pi \alpha,
\end{equation}
joins the origin with the given point and is abnormal. The formula for the distance in Appendix~\ref{a:vertical} yields $d((0,0),(x,Y))^2 = 4\pi \alpha = L(\gamma)^2$, yielding the minimality property of $\gamma$. By Theorem~\ref{t:main}, all elements of $C_k$ are cut points, therefore $t_{\cut}(\gamma) = 1$.
\end{proof}

\section{Small time heat kernel asymptotics}

Any sub-Riemannian manifold equipped with a smooth measure supports an intrinsic hypoelliptic operator, the \emph{sub-Laplacian}, playing the role of the Laplace-Beltrami operator of Riemannian geometry. In the case of Carnot groups equipped with a left-invariant measure -- and in particular for $\G_k$ -- the sub-Laplacian is the ``sum of squares'' given by
\begin{equation}
\Delta := \sum_{i=1}^k X_i^2.
\end{equation}
Let $K_t(q',q)$ denote the \emph{heat kernel} of $\Delta$, that is the fundamental solution of the heat equation $\partial_t \psi = \Delta \psi$. Its existence and smoothness for $t>0$ on complete sub-Riemannian manifolds is classical, see \cite{Strichartz}. A well known result due to L\'eandre \cites{Leandre2,Leandre1} yields the relation between the heat kernel and the sub-Riemannian distance:
\begin{equation}
\lim_{t \to 0^+} 4t \log K_t(q,q') = -d(q,q')^2.
\end{equation}
Furthermore, the singularity of the heat kernel for $t \to 0$ depends on the ``amount'' of minimizing geodesics joining $q$ with $q'$ (see \cite{BBN12} and references therein).

For what concerns Carnot groups of step $2$, an integral representation for $K_t(q):=K_t(q,0)$ has been obtained in \cite[Thm. 4]{BGG-heatkernelstep2} (the general case can be easily recovered by left-invariance). For $\G_k$, and in our notation, such a formula reads
\begin{equation}
K_t(x,Y) = \frac{2}{(4\pi t)^{k^2/2}} \int_{\wedge^2 \R^k} \sqrt{\det\left(\frac{B}{\sin B}\right)} \cos\left( \frac{\tr (B Y^*)}{2t}\right) e^{-\frac{x^*B \cot(B) x}{4t}}\, dB,
\end{equation}
where $d B$ is the Lebesgue measure of $\wedge^2 \R^k \simeq \R^{k(k-1)/2}$, and analytic functions of a matrix are defined by their power series. As an application of our results, combined with the ones in \cite{BBN12}, we mention the following small time asymptotic result.
\begin{theorem}\label{t:heatkernel}
Let $(x,Y) \in C_k \setminus \Abn_0(\G_k)$. Then there exists a constant $C>0$ such that,
\begin{equation}
K_t(x,Y) \geq \frac{C+O(t)}{t^{(n+N(x,Y))/2}}e^{-d^2_0(x,Y)/4t}, \qquad \text{as } t \to 0,
\end{equation}
where $d_0$ is the sub-Riemannian distance from the origin, $n = \dim( \G_k) = k(k+1)/2$, and $N(x,Y)$ is given by Proposition~\ref{p:family} in terms of the normal form of the point $(x,Y)$.
\end{theorem}
\begin{proof}
Let $(p,\Omega)$ be the initial covector corresponding to a minimizing curve with control $u(t) = e^{-\Omega t}p$, joining the origin with $(x,Y) \in C_k$. By Proposition~\ref{p:family}, there exists a continuous family of distinct minimizing geodesics with the same endpoints. The set of initial covectors of this family is 
\begin{equation}
\mathcal{O} = \{(Mp,M\Omega M^*) \mid M \in \mathscr{M}(x,Y)\} \subset T_{(x,Y)}^* \G_k \simeq \G_k,
\end{equation}
where $\mathscr{M}(x,Y)$ is the Lie subgroup of isometries discussed in Section~\ref{s:normalform}. In particular, $\mathcal{O}$ is an $N(x,Y)$-dimensional manifold, and $C_k$ are critical values of the sub-Riemannian exponential map $\exp_0 : T_0^* \G_k \to \G_k$, which sends $(p,\Omega)$ to the point at time $1$ of the curve with control $u(t) = e^{-t\Omega}p$. Moreover, $\dim(\ker d_\lambda \exp_0) \geq N(x,Y)$, for all points $\lambda \in \mathcal{O}$.

Since $(x,Y) \notin \Abn_0(\G_k)$, then all geodesic $\gamma:[0,1] \to M$ of this family are normal. The sub-Riemannian structure on $\G_k$ being real-analytic, then these geodesics are also \emph{strongly normal} (that is, any restriction $\gamma|_{[0,t]}$, for $t \in (0,1]$, is not abnormal, see \cite[Prop. 3.12]{curvature}). A straightforward modification of \cite[Thm. 16]{BBN-biheis} yields the result.
\end{proof}

\section{Closing remarks and open problems}

\subsection{Historical remarks}
The seminal papers by Gaveau \cite{Gaveau77} and Brockett \cite{Brocket80} mentioned in the introduction contain general statements on the minimization problem on $\G_k$ some of which are left unproven. These results might be useful to determine the true shape of $\Cut_0(\G_k)$, and we list them here, commenting their implications.
\begin{itemize}
\item The third (unproven) statement of \cite[Thm. 2]{Brocket80} asserts that any two optimal controls $u_i(t)$, $i=1,2$, such that the associated geodesics $\gamma_i$ join $(0,0)$ with $(x,Y)$ are related by $u_1(t) = \theta u_2(t)$ for some orthogonal matrix such that $\theta Y \theta^* = Y$. We were not able neither to prove nor to disprove this claim. However, if true, it would imply that $\overline{C}_k= \overline{\Cut_0(\G_k)}$ (see Proposition~\ref{p:brockett});
\item By Proposition~\ref{p:family}, points of $C_k$ are critical values of the sub-Riemannian exponential map (see \cites{nostrolibro,riffordbook}). Consequently, any point $(x,Y) \in C_k$ with $x \perp \ker Y$ is a counter-example to the ``only if'' part of the fourth claim of \cite[Thm. 2]{Brocket80};
\item For completeness, we recall that, according to \cite[Appendix A]{LiuSussman}, the result in \cite[Thm. 1, p. 133]{Gaveau77} is false.
\end{itemize}
\begin{prop}\label{p:brockett} 
If the third statement of \cite[Thm. 2]{Brocket80} is true, then $\overline{C}_k = \overline{\Cut_0(\G_k)}$. 
\end{prop}
\begin{proof}
For a general sub-Riemannian manifold, let $\mathrm{II}_q$ be the set of points reached by at least two distinct minimizing geodesics, issued from $q$. By \cite[Thm. 8.54]{nostrolibro}, we have,
\begin{equation}\label{eq:generalfactcut}
\Cut_q \subseteq \overline{\mathrm{II}}_q \cup (\Cut_q \cap \Abn_q).
\end{equation}
Assume that the third statement of \cite[Thm. 2]{Brocket80} is true. In particular, if $(x,Y) \in \mathrm{II}_0(\G_k)$, there exists a non-trivial $\theta \in \mathrm{O}(k)$ that stabilizes $(x,Y)$. Then, the discussion about normal forms of points of $C_k$ in Section~\ref{s:normalform} and the characterization \eqref{eq:abnormalset} imply that $\mathrm{II}_0(\G_k) \subseteq C_k \cup \Abn_0(C_k)$. Taking the closure, and using Proposition~\ref{p:abnormal}, $\overline{\mathrm{II}}_0(\G_k) \subseteq \overline{C}_k$. This, together with \eqref{eq:generalfactcut}, imply $\Cut_0(\G_k) \subseteq \overline{C}_k \subseteq \overline{\Cut_0(\G_k)}$, concluding the proof.
\end{proof}

\subsection{The definition of cut locus}

The sub-Riemannian cut locus is sometimes defined in the literature as the set of points joined by two distinct minimizing geodesics, see e.g.\ \cites{AM-Stiefel,AM-Htypecut,AAA-Tangent}. While this definition might be useful for some statements, it does not correspond to the classical one in the Riemannian case.

\subsection{Codimension}
Generically, one would expect that the cut locus of a complete sub-Riemannian structure has Hausdorff codimension $1$ (with respect to some, and then any, auxiliary Riemannian metric). This is indeed the case if there are no abnormal minimizing curves \cite{RiffTrel}. However, Carnot groups are certainly not ``generic'', and in fact $\codim(\Cut_0(\G_k)) = 2$ for $k=2,3$. By Theorem~\ref{t:main}, $\codim(\Cut_0(\G_k)) \leq 2$ for $k\geq 4$. It is reasonable to expect that $\codim(\Cut_0(\G_k)) =2$ for all $k$, but this problem remains open.

\subsection{Symmetries and geodesic equations}

Note that one can prove that $\Cut_0(\G_2) = C_2$ with a purely symmetry argument. Indeed, the invariance of the cut locus by dilations and isometries implies that if $(x,z) \in \Cut_0(\G_2)$ and $x \neq 0$, then $(\varepsilon M x, \varepsilon^2 z) \in \mathrm{\Cut}_0(\G_2)$, for all $\varepsilon>0$ and $M \in \mathrm{O}(2)$. This parabola separates $\G_2$ into two disconnected regions. Since no minimizing path can cross the parabola at an intermediate time, this yields a contradiction, and implies that $\mathrm{Cut}_0(\G_2) \subseteq \{(x,0)\mid x \in \R^2_*\} = C_2$. By Theorem~\ref{t:main} -- the proof of which, we stress, never uses the integrated geodesic equations -- we have $C_2 \subseteq \Cut_0(\G_2)$, which proves the claim.

Unfortunately, this argument does not carry on for rank $k \geq 3$. In order to prove whether $C_k = \Cut_0(\G_k)$, for all $k$, it seems that a fine analysis of the integrated geodesic equations is required (as for the cases $k=2,3$). Any attempt in this direction had proved inconclusive, due to the very complicated structure of the geodesics for $k \geq 4$.

\appendix

\section{Distance from the points on the vertical subspace}\label{a:vertical}

We compute the distance from the origin of vertical points in $\G_k$. A very close formula appears as the second statement of \cite[Thm. 2]{Brocket80}, and differs from ours by a factor $4\pi$. 
\begin{prop}
Let $(0,Y) \in \G_k$, with $\rank (Y) = 2r$, and let $0<\alpha_1\leq \alpha_2 \leq \cdots \leq \alpha_r$ be the (possibly repeated) absolute values of the non-zero eigenvalues of $Y$. Then,
\begin{equation}
d((0,0),(0,Y))^2 = 4 \pi \sum_{j=1}^r (r-j+1) \alpha_j.
\end{equation}
\end{prop}
\begin{proof}
Without loss of generality, 
Let $\gamma(t) =(x(t),Y(t))$ a geodesic from the origin such that $x(1) = 0$ and $Y(1) = Y$, with control $u(t) = e^{-\Omega t}p$. By \eqref{eq:geods}, we have
\begin{equation}
\int_0^1 e^{-t\Omega} p \,dt =0.
\end{equation}
Thus, the non-zero eigenvalues of $\Omega$ are of the form $\pm i 2\pi \phi$, with $\phi \in \N$. By Remark~\ref{r:gauthier}, and up to an orthogonal transformation, we assume that $\Omega = \diag(2\pi \phi_1\mathbbold{J}_2, \dots, 2\pi \phi_\ell \mathbbold{J}_2, \mathbbold{0}_{k - 2\ell})$, with all simple eigenvalues, $2\ell = \rank(\Omega)$, and with distinct $\phi_i \in \N$. We split accordingly $p=(p_1,\dots,p_\ell,p_0)$, with $p_i \in \R^2$ for $i=1,\dots,\ell$ and $p_0 \in \R^{k-2\ell}$. Under these assumptions, it is not hard to integrate the vertical part of the geodesic equations \eqref{eq:geods}. We obtain
\begin{equation}
Y(1) = \diag\left(\frac{|p_1|^2}{4\pi \phi_1}, \dots, \frac{|p_\ell|^2}{4\pi \phi_\ell}, \mathbbold{0}_{k-2\ell} \right).
\end{equation}
Then, $2\ell = 2r$, and $|p_j|^2 = 4 \pi \phi_j  \alpha_j$ for all $j=1,\dots,r$. The squared length of $\gamma$ is
\begin{equation}
L(\gamma)^2 = \left(\int_0^1 |u(t)| dt\right)^2 = |p|^2 = \sum_{j=1}^r |p_j|^2 = 4\pi \sum_{j=1}^r \phi_j \alpha_j.
\end{equation}
The minimum of the length is obtained when $\phi_j = r-j+1$, for all $j=1,\dots,r$.
\end{proof}

\subsection*{Acknowledgments}
This research has  been supported by the European Research Council, ERC StG 2009 ``GeCoMethods'', contract n. 239748. The first author was partially supported by the Grant ANR-15-CE40-0018 of the ANR, by the iCODE institute (research project of the Idex Paris-Saclay), and by the SMAI project ``BOUM''. This research benefited from the support of the ``FMJH Program Gaspard Monge in optimization and operation research'' and from the support to this program from EDF. The second author was partially supported by the Grant ANR-12-BS03-0005 LIMICOS of the ANR.

\medskip

The first author wishes to thank A. Gentile for helpful discussions on semi-algebraic sets, U. Boscain for hinting at the implications of Proposition~\ref{p:family} to heat kernel asymptotics, and Y. Sachkov for carefully reading a preliminary version of the paper.

\bibliographystyle{alpha}
\bibliography{cut-free}

\end{document}